\documentclass[12pt]{article}

\usepackage{hyperref}

\usepackage{url}
\usepackage{pslatex}
\usepackage{amsmath,amsthm,amssymb,amsfonts, stmaryrd}
\usepackage{fancybox}
\usepackage[T1]{fontenc} 
\usepackage[usenames,dvipsnames,svgnames,table]{xcolor}

\usepackage{xcolor, graphicx}

\usepackage{dsfont}

\newtheorem{theorem}{Theorem}[section]
\newtheorem{lemma}[theorem]{Lemma}
\newtheorem{proposition}[theorem]{Proposition}
\newtheorem{corollary}[theorem]{Corollary}
\newtheorem{remark}[theorem]{Remark}
\newtheorem{notation}[theorem]{Notation}
\newtheorem{definition}[theorem]{Definition}

\renewcommand{\a}{\alpha}

\newcommand{\N}{\mathbb N}

\newcommand{\Z}{\mathbb Z}
\newcommand{\vs}[1]{\langle #1 \rangle}

\newcommand{\ap}{{\rm Ap}}

\newcommand{\sbu}{{\small $\bullet$ }}
\newcommand{\sbbu}{{\small $\bullet\bullet$ }}

\newcommand{\ind}{\mathbf{1}}

\title{On numerical semigroups with at most 12 left elements}

\author{S. Eliahou and D. Mar\'in-Arag\'on}

\date{}

\begin{document}

\maketitle

\begin{abstract} For a numerical semigroup $S \subseteq \N$ with embedding dimension $e$, conductor $c$ and left part $L=S \cap [0,c-1]$, set $W(S) = e|L|-c$. In 1978 Wilf asked, in equivalent terms, whether $W(S) \ge 0$ always holds, a question known since as Wilf's conjecture. Using a closely related lower bound $W_0(S) \le W(S)$, we show that if $|L| \le 12$ then $W_0(S) \ge 0$, thereby settling Wilf's conjecture in this case. This is best possible, since cases are known  where $|L|=13$ and $W_0(S)=-1$. Wilf's conjecture remains open for $|L| \ge 13$. 
\end{abstract}

\section{Introduction}

Let $\N=\{0,1,2,\ldots \}$ denote as usual the set of nonnegative integers. Given integers $a \le b$, we denote by $[a,b[=[a,b-1]=\{z \in \Z \mid a \le z < b\}$, and by $[a,\infty[=\{z \in \Z \mid a \le z\}$, the integer intervals they span. A \emph{numerical semigroup} is a submonoid $S$ of $(\N,+)$ such that $|\N\setminus S|$ is finite. Equivalently, it is a subset $S$ of $\N$ of the form $S = \vs{a_1,\dots,a_n}=\N a_1+\dots+\N a_n$ where $\gcd(a_1,\dots,a_n)=1$. The least such $n$ is called the \emph{embedding dimension} of $S$ and is often denoted $e$. The \emph{multiplicity} of $S$ is $m = \min S^*$, where $S^*=S \setminus \{0\}$. The \emph{conductor} of $S$ is $c=\max (\Z \setminus S)+1$, or equivalently, the least $c \in \N$ such $[c, \infty[ \, \subseteq S$. The \emph{genus} of $S$ is $g = |\N \setminus S|$. The \emph{left part} of $S$ is 
$$L=\{s\in S \mid s<c\} = S \cap [0,c[.$$ 
The \emph{left elements of $S$} are the elements of $L$. Finally, as in \cite{WilfMacaulay}, we denote
\begin{equation}\label{W}
W(S) = e|L|-c.
\end{equation} 
In 1978 Wilf asked, in equivalent terms, whether the inequality
$$W(S) \ge 0$$
holds for every numerical semigroup $S$~\cite{Wilf}. This open question is now known as Wilf's conjecture. Various particular cases have been settled, including the six independent cases $e \le 3$, $|L| \le 6$, $m \le 18$, $g \le 60$, $c \le 3m$ and $e \ge m/3$. See e.g. \cite{BGOW, D, cota4, WilfMacaulay, E2, FGH, FH, K, Mosca, Mio, Sammartano, Sy}. See also~\cite{D2} for a recent extensive survey on this topic.

The authors of \cite{cota4} settled Wilf's conjecture in case $|L|\leq 4$. This was later extended in~\cite{WilfMacaulay}, where a certain lower bound $W_0(S) \le W(S)$ was introduced and shown to satisfy $W_0(S) \ge 0$ whenever $|L|\leq 6$. Here we further extend this result by showing that $W_0(S)\geq 0$ holds whenever $|L|\leq 12$. This is best possible since, as shown in~\cite{Jean}, there are numerical semigroups $S$ such that $|L| = 13$ and $W_0(S) < 0$. See also Section~\ref{sub W0}.

This paper is organized as follows. In Section~\ref{section notation} we recall some notation and background, including the definition of $W_0(S)$. In Section~\ref{section profile} we establish $W_0(S) \ge 0$ in some special circumstances. Our main result, namely $W_0(S) \ge 0$ if $|L| \le 12$, and hence Wilf's conjecture in that case, is proved in Section~\ref{main}.

For extensive information on numerical semigroups, see \cite{libro}.

\section{Background and notation}\label{section notation}
In this section, we recall some notation and terminology introduced in \cite{WilfMacaulay}. Let $S$ be a numerical semigroup. We denote by $P \subset S^*$ the unique minimal generating set of $S$, so that $S = \vs{P}$ and $|P|=e$, the embedding dimension. It coincides with the set of \emph{primitive elements} of $S$, i.e. those $x \in S^*$ which are not the sum of two smaller elements of $S^*$. Let $m,c$ be the multiplicity and conductor of $S$, respectively. The \emph{depth} of $S$ is $q = \lceil c/m \rceil$ and its \emph{offset} is
$\rho = qm-c$. Thus
\begin{equation}\label{rho}
c=qm-\rho, \ \ \rho \in [0,m[.
\end{equation}
The set of \emph{decomposable elements} of $S$ is $$D=S^* \setminus P=S^*+S^*.$$
Note that $D$ contains $[c+m,\infty[$. Indeed, if $z \ge c+m$, then $z=m+(z-m)$, so that  $z \in S^*+S^*$ since both $m \in S^*$ and $z-m \in S^*$ as $z-m \ge c$. It follows that
\begin{equation}\label{P confined}
P \subseteq [m,c+m[.
\end{equation}

Throughout Section~\ref{section notation}, the symbols $m,c,q$ and $\rho$ will denote, often tacitly so, the multiplicity, conductor, depth and offset of the numerical semigroup $S$ under consideration, respectively.

\subsection{The level function $\lambda$}

Let $S \subseteq \N$ be a numerical semigroup. We shall further use the following notation, as in~\cite{WilfMacaulay}. 

\begin{notation}\label{nota} For all $j \in \Z$, we denote
\begin{eqnarray*}
I_j & = & [jm-\rho,(j+1)m-\rho[, \\
S_j & = & S\cap I_j,\,\, P_j=P\cap I_j,\,\, D_j=D\cap I_j.
\end{eqnarray*}
\end{notation}
In particular, we have
$$I_q=[qm-\rho, (q+1)m-\rho[ \ = [c,c+m[.$$

The following set addition rules are shown in \cite{WilfMacaulay}. The proof is straightforward and left to the reader.
\begin{proposition}\label{addition} Let $S$ be a numerical semigroup. For all $i,j \ge 1$, we have \vspace{-0.2cm}
\begin{eqnarray*}
S_1 + S_j & \subseteq & S_{1+j} \cup S_{2+j}, \\
S_i + S_j & \subseteq & S_{i+j-1} \cup S_{i+j} \cup S_{i+j+1} \ \textrm{ if } i,j \ge 2. \hspace{1cm}\Box
\end{eqnarray*} 
\end{proposition}

\begin{notation} The \emph{level function} $\lambda_S \colon \N \to \N$ associated to $S$ is defined by
$$
\lambda_S(x) = j \iff x \in I_j \iff jm-\rho \le x \le (j+1)m - \rho -1
$$
for all $x \in \N$. In particular, if $x \in S$ then $\lambda_S(x) = j$ if and only if $x \in S_j$.
\end{notation}
In the sequel, for simplicity, we shall write $\lambda$ for $\lambda_S$. Using this function, the above proposition translates as follows.

\begin{corollary}\label{cor level} Let $x,y \in S^*$. If $\lambda(x), \lambda(y) \ge 2$ then
\begin{equation}\label{level}
\lambda(x)+\lambda(y)-1 \le \lambda(x+y) \le \lambda(x)+\lambda(y)+1.
\end{equation}
If $\lambda(x) =1$ or $\lambda(y) =1$, then 
$\lambda(x)+\lambda(y) \le \lambda(x+y) \le \lambda(x)+\lambda(y)+1$. \hfill $\Box$
\end{corollary}

\medskip
Here are some more consequences.
\begin{corollary}  Let $a,x,y \in S^*$. Then $\lambda(x+y) > \max\{\lambda(x),\lambda(y)\}$. If $\lambda(a+x)=\lambda(a+y)$, then $|\lambda(y)-\lambda(x)| \le 1$.
\end{corollary}
\begin{proof} The first statement directly follows from Corollary~\ref{cor level}. As for the second one, let $i=\lambda(x), j=\lambda(y)$. We may assume $i \le j$. Let $k=\lambda(a+x)$. Then $a+x, a+y \in S_k$. Hence $|y-x|=|(a+y)-(a+x)| \le m-1$. It follows that $j \le i+1$, since if $j \ge i+2$ then $\min S_j - \max S_i \ge m+1$.
\end{proof}

\subsection{The number $W_0(S)$}\label{sub W0}
\begin{notation} For a numerical semigroup $S$, we denote
\begin{equation}\label{W0}
W_0(S) = |P\cap L||L|-q|D_q|+\rho.
\end{equation}
\end{notation}
Introduced in~\cite{WilfMacaulay}, this number bounds $W(S)$ from below and is sometimes easier to evaluate. See also~\cite{D}, where $W_0(S)$ is denoted $E(S)$. The following result is Proposition 3.11 in~\cite{WilfMacaulay}. For convenience, we recall the short proof.

\begin{proposition}
Let $S$ be a numerical semigroup. Then $W(S)\geq W_0(S)$.
\end{proposition}
\begin{proof} We have $W(S)=|P||L|-c=|P||L|-qm+\rho$. We have $m=|P_q|+|D_q|$, since $m=|S_q|=|P_q \sqcup D_q|$, and $|P|=|P \cap L|+|P_q|$. It follows that
$$
W(S)=W_0(S)+|P_q|(|L|-q).
$$
Now $|L| \ge q$, since $L$ contains the $q$-subset $\{0,1,\dots,q-1\}m$.
\end{proof}

\begin{corollary}\label{Wilf with W0}
Let $S$ be a numerical semigroup such that $W_0(S) \ge 0$. Then $S$ satisfies Wilf's conjecture. 
\end{corollary}
\begin{proof} We have $W(S) \ge W_0(S) \ge 0$.
\end{proof}
This corollary is the basis of our approach in this paper, whose main result is that  $W_0(S) \ge 0$ whenever $|L| \le 12$. Note that in contrast to Wilf's conjecture, the number $W_0(S)$ can be negative, but such cases are extremely rare. For instance, among the more than $10^{13}$ numerical semigroups of genus $g \le 60$, only five of them satisfy $W_0(S) < 0$. See \cite{D, Jean, FH}. More specifically, these five exceptions all satisfy $W_0(S)=-1$, $|L|=13$ and $c=4m$, and they occur at genus $43, 51, 55, 55$ and $59$, respectively. The first one, of genus $g=43$, is $S=\vs{14,22,23} \cup [56, \infty[$.

\medskip

The following result has been established in~\cite{WilfMacaulay}.

\begin{theorem}\label{q le 3} Let $S$ be a numerical semigroup of depth $q \le 3$. Then $W_0(S) \ge 0$. In particular, $S$ satisfies Wilf's conjecture.
\end{theorem}
Consequently, in proving here that $W_0(S) \ge 0$ if $|L| \le 12$, we only need to consider the case of depth $q \ge 4$. The next three sections focus on the Apéry set of $S$ with respect to $m$ and provide tools to evaluate $W_0(S)$ and prove our main result.

\subsection{The Ap\'ery profile of $S$}\label{sub apery}
Let $S \subseteq \N$ be a numerical semigroup. We denote by $$A =\ap(S,m)=\{x\in S \mid x-m\not\in S\}$$ the \emph{Ap\'ery set} of $S$ with respect to $m$. Equivalently, $A = S \setminus (m+S)$. Each Apéry element $x \in A$ is the smallest element in $S$ of its class mod $m$, since $x-m \notin S$. Hence $|A|=m$. We have
$$A \subseteq [0, c+m[.$$ 
Indeed, this follows from the inclusion $[c+m,\infty[ \: = m + [c,\infty[ \: \subseteq m+S$. We now introduce the Apéry profile of $S$.
\begin{notation}\label{Ai} Let $S$ be a numerical semigroup of depth $q \ge 1$. For all $0\leq i\leq q$, we set $A_i  = A\cap I_i$ and $$\a_i = |A_i|.$$
\end{notation}
We have $A_0=\{0\}$, so $\a_0=1$. Moreover, $P_1=\{m\} \sqcup A_1$ and $P_i \subseteq A_i$ for all $i \ge 2$.
\begin{definition}
We call \emph{Apéry profile} of $S$ the $(q-1)$-tuple 
$$\a(S) = (\a_1,\dots,\a_{q-1}) \in \N^{q-1}.$$
\end{definition}
As noted above, we have
\begin{equation}\label{alpha i}
\a_1=|P_1|-1 \textrm{ and }\, \a_i\geq |P_i| \, \textrm{ for all } 2\leq i< q.
\end{equation}
Moreover, since $A \subseteq [0, c+m[ = I_0 \sqcup I_1 \sqcup \dots \sqcup I_q$, we have
\begin{equation}\label{part A}
A = A_0 \sqcup A_1 \sqcup \dots \sqcup A_q.
\end{equation}
Therefore
\begin{equation}\label{eq sum a_i}
\sum_{i=0}^q \a_i = |A| = m.
\end{equation}
This justifies why $\a_0,\a_q$ are not included in the profile $\a(S)$, as $\a_0=1$ and $\a_q$ may be recovered from $\a(S)$ and $m$ by the above formula.

\subsection{Primitive and decomposable Apéry elements}
Let $S$ be a numerical semigroup of multiplicity $m$ and Apéry set $A=S \setminus (m+S)$. A key point in the sequel is to distinguish, in $A^* = A \setminus \{0\}$, the primitive elements from the decomposable ones. Indeed, the partition
$$A^* = (A \cap P) \sqcup (A \cap D)$$
plays an important role and motivates the following notation.

\begin{notation}\label{prim and dec} Let $S$ be a numerical semigroup of depth $q \ge 1$. For all  $1 \le i \le q$, we set 
$$\a_i' = |A_i \cap P|, \ \ \a_i'' = |A_i \cap D|.$$ 
Thus $\alpha_i = \alpha_i'+\alpha_i''$ for all $i$.
\end{notation}

Since $P \subseteq [m,c+m[$ as seen above, and since $[m,c+m[ \ \subseteq I_1 \cup \dots \cup I_q$, we have
\begin{equation}\label{card of P}
|P| = 1+\a'_1 + \dots + \a'_{q}.
\end{equation}
In particular,
\begin{eqnarray*}
\a'_{q} & = & |P \setminus L| \, =\,  \big|P \cap [c,c+m[\big|, \\
|D_q| & = & |D \cap [c,c+m[| \, =\, \big|[c,c+m[ \setminus P\big |, \\  
m & = & |P_q| +|D_q|.
\end{eqnarray*}

The following properties of the Apéry set $A$ will be widely used below, often tacitly so.

\begin{lemma}\label{downset} Let $z \in A \cap D$. If $z =x+y$ with $x,y \in S^*$, then $x,y \in A^*$.
\end{lemma}
\begin{proof} If $x \notin A$, then $x=m+s$ for some $s \in S$, implying $z=x+y=m+(s+y)$. Since $s+y \in S$, it follows that $z\notin A$, contrary to the hypothesis.
\end{proof}

\begin{proposition} For all $k \ge 2$, we have
$\displaystyle
A_k \cap D \subseteq \cup_{i,j} (A_i + A_j)
$
where $1 \le i \le j$ and $k-1 \le i+j \le k+1$. 
\end{proposition}
\begin{proof} Directly follows from Proposition~\ref{addition} and Lemma~\ref{downset}.
\end{proof}

\begin{corollary} If $A_k \not= \emptyset$ and $A_i = \emptyset$ for all $1 \le i \le k-1$ for some $k \ge 2$, then $A_k = P_k$.
\end{corollary}
\begin{proof} Directly follows from the above proposition.
\end{proof}

\subsection{Compressed Apéry elements}
Throughout, let $S$ denote a numerical semigroup with multiplicity $m$, conductor $c$ and Apéry set $A=S \setminus (m+S)$.

\begin{definition} Let $s \in S^*$. We say that $s$ is \emph{compressed} if there exist $x,y \in S^*$ such that $s=x+y$ and $\lambda(s) < \lambda(x)+\lambda(y)$.
\end{definition}
By Corollary~\ref{cor level}, we have $\lambda(x+y) \ge \lambda(x)+\lambda(y)-1$ for all $x,y \in S^*$. Thus, the inequality $\lambda(x+y) < \lambda(x)+\lambda(y)$ is equivalent to $\lambda(x+y) = \lambda(x)+\lambda(y)-1$. Estimating the number of compressed elements in $A$ is important in the sequel. This motivates the following notation.

\begin{notation}\label{C}
$C = C(S)=\{z \in A\cap D \mid z \textrm{ is compressed}\}$.
\end{notation}
Recall from Lemma~\ref{downset} that if $z \in A\cap D$ and $z=x+y$ with $x,y \in S^*$, then in fact $x,y \in A^*$. Consequently, for  all $i,j \ge 2$, we have
\begin{equation}\label{on C}
(S_i+S_j)\cap A_{i+j-1} = (A_i+A_j)\cap A_{i+j-1} \ \subseteq \ C.
\end{equation}
More generally, even if the description below will not be needed here, we have
$$
C = \bigcup_{k=3}^q A_k \cap (\cup_{i=1}^{k-1} (A_i+A_{k+1-i})).
$$
The next result provides a key lower bound on $\rho$, where $\rho$ is the offset as defined in \eqref{rho}. See also Proposition 3.20 in \cite{E2}.
\begin{proposition}\label{deficit} Let $S$ be a numerical semigroup. Then $\rho \ge |C|$. 
\end{proposition}
\begin{proof} Let $z \in C$, and assume $z = x+y$ with $x,y \in A^*$ such that $\lambda(z)=\lambda(x)+\lambda(y)-1$. Say
$\lambda(x)=i$, $\lambda(y)=j$ and $\lambda(z)=i+j-1$. By the definition of $S_i$, we have
$$(S_i+S_j) \cap S_{i+j-1} \subseteq [(i+j)m-2\rho,(i+j)m-\rho[.
$$
Thus $z \in [(i+j)m-2\rho,(i+j)m-\rho[$. Now, the only classes mod $m$ occurring in the latter interval are those in $[-2\rho,-\rho[$, a set of cardinality $\rho$. Since there is only one element in $A$ per class mod $m$, and since $C \subset A$, the statement follows.
\end{proof}

In particular, we shall invoke the following simplified version.

\begin{corollary}\label{cor rho} For all $i,j \ge 2$, we have $\rho \ge |(A_i+A_j)\cap A_{i+j-1}|$.
\end{corollary}
\begin{proof}
Follows from \eqref{on C} and Proposition~\ref{deficit}.
\end{proof}

\subsection{Computing $W_0(S)$}

The following formulas allow to evaluate $W_0(S)$ using the Apéry profile of $S$ as defined in~Definition~\ref{Ai} and the decomposition $\a_q=\a_q'+\a_q''$ given by~Notation~\ref{prim and dec}. Recall that both $|L|$ and $|D_q|$ are involved in the expression of $W_0(S)$.
\begin{proposition}\label{prop formulas} Let $S \subseteq \N$ be a numerical semigroup with Apéry profile $\a(S)=(\a_1,\dots,\a_{q-1})$. Then
\begin{eqnarray*}
|L| & = & q +(q-1)\a_1+ \dots + \a_{q-1}, \\
|D_q| & = & \a_0+\a_1+\dots+\alpha_{q-1}+\alpha_q''.
\end{eqnarray*}
\end{proposition}
\begin{proof} 
Let $s \in L$ be minimal in its class mod $m$. Then $s \in L \cap A$. Let $i \ge 0$ be the unique integer such that $s \in A_i$. Then $0 \le i \le q-1$ since $s \in L$. Let $z \in L$ be such that $z \equiv s \bmod m$. Then $z = s +jm$ with $0 \le j \le q-1-i$. It follows that
$$
|L \cap (s+m\N)| = q-i.
$$
Letting now $s$ run through all elements of $L$ which are minimal in their respective classes mod $m$, the above discussion implies
$$|L|=q|A_0|+(q-1)|A_1|+ \dots +|A_{q-1}|,$$ 
yielding the first formula. Since $m=|P_q|+|D_q|=\a_q'+|D_q|$, it follows that 
$$|D_q|=m-\a_q'=m-\a_q+\a_q''.$$ The second formula now follows from~\eqref{eq sum a_i}, i.e. $m=\sum_{i=0}^q \a_i$.
\end{proof}

\subsection{Notation}\label{subsec notation}
We shall use the following notation throughout the remainder of the paper. Given a numerical semigroup $S$, we denote by $m$ its multiplicity, by $c$ its conductor, by $q$ its depth, by $\rho$ its offset, by $L$ its left part, by $P$ its set of primitive elements, by $D$ its set of decomposable elements, by $A = S \setminus (m+S)$ its Apéry set with respect to $m$, and by 
$$\a(S) = (\a_1,\dots,\a_{q-1})$$
its Apéry profile, where $\a_i=|A_i|$ for all $0 \le i \le q$. For $i \ge 1$, we have $\a_i=\a_i'+\a_i''$ where $\a_i'=|A_i \cap P|$ and $\a_i''=|A_i \cap D|$. We shall constantly use the formulas below to compute the ingredients involved in 
$$W_0(S) = |P \cap L||L|-q|D_q|+\rho,$$ namely
\begin{equation}
|P \cap L| = 1 + \sum_{i=1}^{q-1} \a_i', \ \ 
|L| = \sum_{i=0}^{q-1} (q-i)\a_i , \ \
|D_q| = \sum_{i=0}^{q-1} \a_i + \a_q''.
\end{equation}

Finally, as in the preceding section, we denote by $C$ the set of compressed Apéry elements, i.e
$$
C = \{z \in A\cap D \mid \exists x,y \in A^*, \, z=x+y,\, \lambda(z) < \lambda(x)+\lambda(y)\}.
$$
In a few cases, the estimate $\rho \ge |C|$ provided by Proposition~\ref{deficit} will be crucially needed in order to be able to conclude $W_0(S) \ge 0$.

\section{An occurrence of $W_0(S) \ge 0$}\label{section profile}

In this section, we establish $W_0(S) \ge 0$ for numerical semigroups $S$ under suitable assumptions on $\a(S)$ but not on $|L|$. We use the notation of Section~\ref{subsec notation} throughout.

\begin{theorem}\label{thm:h} Let $S$ be a numerical semigroup of depth $q \ge 4$. Let $h = \lceil q/2 \rceil$. Assume that $\alpha_i=0$ for all $1 \le i \le h-1$. Then $W_0(S) \ge 0$.
\end{theorem}
\begin{proof} Since $P = \cup_{1 \le i \le q}$ and $P_1=\{m\} \sqcup A_1$, we have $|P|=1+\a_1+\a'_2+\dots+\a'_{q}$ and
\begin{equation}\label{P cap L}
|P \cap L| = 1+\a_1+\a'_2+\dots+\a'_{q-1}.
\end{equation}

\medskip
\noindent
\textbf{\sbu Assume $q$ odd.} Then $h = (q+1)/2$. Since $A_1=\dots=A_{h-1}=\emptyset$, we have
$$
A \cap D \subseteq \bigcup_{i,j=h}^{q-1} (A_i+A_j).
$$
By Proposition~\ref{addition}, we have
$
A_i+A_j \subseteq S_{i+j-1} \sqcup S_{i+j} \sqcup S_{i+j+1}.
$
Since $A_t = \emptyset$ for $t \ge q+1$, and since $2h=q-1$, it follows from the above that 
$$
A \cap D \subseteq 2A_h.
$$
Hence $\a_q'' \le |2A_h|$. It also follows that $A_i=P_i$ for all $h \le i \le q-1$. Hence
\begin{eqnarray*}
|P \cap L| & = & 1+\a_h+\dots+\a_{q-1} \\
& \ge & 1+\a_h, \\
|L| & = & q + (h-1)\a_h+(h-2)\a_{h+1}+\dots+\a_{q-1} \\
& \ge & q + (h-1)\a_h, \\
|D_q| & = & 1+\a_h+\dots+\a_{q-1}+\a_q'' \\
& = & |P \cap L|+\a_q''.
\end{eqnarray*}
Thus,
\begin{eqnarray*}
W_0(S) & = & |P \cap L||L|-q|D_q|+\rho \\
& = & |P \cap L||L|-q(|P \cap L|+\a_q'')+\rho \\
& = & |P \cap L|(|L|-q)-q\a_q''+\rho \\
& = & (1+\a_h)(h-1)\a_h-q\a_q''+\rho.
\end{eqnarray*}

Since $A \cap D_q \subseteq 2A_h$, and since $q=2h-1$, it follows from Lemma~\ref{deficit} that $A_q \cap D \subseteq C$, where $C$ is the subset defined in that Lemma. Consequently, we have $|A_q \cap D| \le |C| \le \rho$, whence
$$
\a_q'' \le \min(|2A_h|, \rho).
$$
Therefore
$$
W_0(S) \ge (1+\a_h)(h-1)\a_h-q\min(|2A_h|, \rho)+\rho.
$$
The following bound will take care of the last two summands. 

\medskip
\noindent
\textbf{Claim.} For all $t \ge 0$, we have
\begin{equation}\label{easy}
-q\min(t,\rho)+\rho \ \ge \ -(q-1)t.
\end{equation}
Indeed, if $\rho \le t$, then $-q\min(t,\rho)+\rho = -(q-1)\rho \ge -(q-1)t$. And if $\rho > t$, then $-q\min(t,\rho)+\rho = -qt+\rho > -(q-1)t$. This proves the claim.

\smallskip
Moreover, as a very crude estimate, we have
$$
|2A_h| \le \a_h(\a_h+1)/2.
$$
Hence, using \eqref{easy} and the above, we get 
$$
W_0(S) \ge (1+\a_h)(h-1)\a_h-(q-1)\a_h(\a_h+1)/2.
$$
Since $(q-1)/2=h-1$, it follows that $W_0(S) \ge 0$, as desired.

\medskip
\noindent
\textbf{\sbu Assume $q$ even.} Then $h = q/2$. Since $A_1=\dots=A_{h-1}=\emptyset$, we have
$$
A \cap D \subseteq \bigcup_{i,j=h}^{q-1} (A_i+A_j).
$$
By Proposition~\ref{addition}, we have
$
A_i+A_j \subseteq S_{i+j-1} \sqcup S_{i+j} \sqcup S_{i+j+1}.
$
Since $A_t = \emptyset$ for $t \ge q+1$, and since $2h=q$, it follows from the above that 
$$
A \cap D \subseteq (2A_h) \cup (A_h+A_{h+1}).
$$
Moreover, we have
$$
2A_h \cap A \subseteq A_{q-1} \cup A_q, \quad (A_h+A_{h+1}) \cap A \subseteq A_{q}.
$$
Hence
$$
A \cap D \subseteq (2A_h \cap A_{q-1}) \cup (2A_h \cap A_{q}) \cup (A_h+A_{h+1}) \cap A_{q}. 
$$
We have
\begin{eqnarray*}
\a_{q-1}'' & = & |2A_h \cap A_{q-1}|, \\
\a_q'' & = & |(2A_h \cap A_{q}) \cup (A_h+A_{h+1}) \cap A_{q}|.
\end{eqnarray*}
Hence 
\begin{equation}\label{q-1 et q}
\a_{q-1}''+\a_q'' \le |2A_h|+|(A_h+A_{h+1}) \cap A_{q}|.
\end{equation}
Moreover, since $q=2h$, Corollary~\ref{cor rho} yields 
$$
|(A_h+A_{h+1}) \cap A_{q}| \le \rho.
$$
Moreover, we have $|(A_h+A_{h+1}) \cap A_{q}| \le \a_h\a_{h+1}$. Hence 
\begin{equation}\label{h et h+1}
|(A_h+A_{h+1}) \cap A_{q}| \le \min(\rho, \a_h\a_{h+1}).
\end{equation}
Combining \eqref{q-1 et q} and \eqref{h et h+1}, we get
\begin{equation}\label{combin}
\a_{q-1}''+\a_q'' \le |2A_h|+\min(\rho, \a_h\a_{h+1}).
\end{equation}
It also follows that $A_i=P_i$ for all $h \le i \le q-2$. Hence
\begin{eqnarray*}
|P \cap L| & = & 1+\a_h+\dots+\a_{q-2}+\a_{q-1}' \\
& \ge & 1+\a_h+\a_{h+1}, \\
|L| & = & q + h\a_h+(h-1)\a_{h+1}+\dots+\a_{q-1} \\
& \ge & q + h\a_h+(h-1)\a_{h+1}, \\
|D_q| & = & 1+\a_h+\dots+\a_{q-1}+\a_q'' \\
& = & |P \cap L|+\a_{q-1}''+\a_q''.
\end{eqnarray*}
Thus,
\begin{eqnarray*}
W_0(S) & = & |P \cap L||L|-q|D_q|+\rho \\
& = & |P \cap L||L|-q(|P \cap L|+\a_{q-1}''+\a_q'')+\rho \\
& = & |P \cap L|(|L|-q)-q(\a_{q-1}''+\a_q'')+\rho \\
& \ge & (1+\a_h+\a_{h+1})(h\a_h+(h-1)\a_{h+1}-q(\a_{q-1}''+\a_q'')+\rho.
\end{eqnarray*}
Therefore, using \eqref{combin}, we get
$$
W_0(S) \ge (1+\a_h+\a_{h+1})(h\a_h+(h-1)\a_{h+1})-q(|2A_h|+\min(\rho, \a_h\a_{h+1}))+\rho.
$$
Using \eqref{easy} again for the last two summands, we have 
$$-q\min(\rho, \a_h\a_{h+1})+\rho \ge -(q-1)\a_h\a_{h+1.}$$ We also have 
the very crude estimate
$$
|2A_h| \le \a_h(\a_h+1)/2.
$$
Hence
$$
W_0(S) \ge (1+\a_h+\a_{h+1})(h\a_h+(h-1)\a_{h+1})-q\a_h(\a_h+1)/2-(q-1)\a_h\a_{h+1}.
$$
Using $q=2h$, it follows that 
$$W_0(S) = (h-1)\a_{h+1}(\a_{h+1}+1).$$ 
Hence $W_0(S) \ge 0$, as desired.
\end{proof}

We conclude this section with an easy particular case.

\begin{proposition} Let $S$ be a numerical semigroup of depth $q \ge 4$ such that $|P \cap L| \ge \max(\a''_q,q)$. Then $W_0(S) \ge \rho$.
\end{proposition}
\begin{proof} $W_0(S) = |P \cap L||L|-q|D_q|+\rho$. Recall that $|L|=\sum_{i=0}^{q-1} (q-i)\a_i$ and that $|D_q|=\sum_{i=0}^{q-1}\a_i + \a''_q$. Hence
$$
W_0(S) = q(|P \cap L|-\a''_q)+\sum_{i=1}^{q-1}((q-i)|P \cap L|-q)\a_i + \rho.
$$
Since $|P \cap L| \ge \a_q''$ and $|P \cap L| \ge q$ by hypothesis, the claimed inequality follows.
\end{proof}

\section{Main result}\label{main}

Let $S$ be a numerical semigroup. We use the notation of Section~\ref{subsec notation} throughout. Wilf's conjecture has been successively settled for $|L| \le 4$ and $|L| \le 6$ in \cite{cota4} and \cite{WilfMacaulay}, respectively. Here we extend these results to the case $|L| \le 12$. Even more so, we show that if $|L| \le 12$ then $W_0(S) \ge 0$. As mentioned earlier, this is best possible, since there are numerical semigroups $S$ satisfying $|L|=13$ and $W_0(S) <0$. At the time of writing, it remains an open problem to determine whether all numerical semigroups $S$ with $|L|=13$ satisfy Wilf's conjecture. In this section we prove the following result.

\begin{theorem}\label{main thm}
Let $S$ be a numerical semigroup such that $|L|\leq 12$. Then $W_0(S)\geq 0$. In particular, $S$ satisfies Wilf's conjecture.
\end{theorem} 

By Theorem~\ref{q le 3}, the bound $W_0(S) \ge 0$ holds for all numerical semigroups of depth $q \le 3$. Consequently, in the sequel, we shall freely assume $q \ge 4$, since it suffices to prove Theorem~\ref{main thm} in that case. In fact, it also suffices to consider the case $q \le 7$, as follows from the following proposition.

\begin{proposition}\label{prop q ge 8}  Let $S$ be a numerical semigroup of depth $q \ge 8$ such that $|L| \le 12$. Then $W_0(S) \ge 0$.
\end{proposition}
\begin{proof} Let $h = \lceil q/2 \rceil$. Then
$
12 \ge |L| \ge q+(q-1)\a_1+\dots+(q-h+1)\a_{h-1}.
$
This implies $\a_i=0$ for all $1 \le i \le h-1$. For if not, then $|L| \ge q+(q-h+1) $, and since $h \le (q+1)/2$, we would get $12 \ge |L| \ge 2q-(q+1)/2+1=(3q+1)/2$ 
and hence $3q+1 \le 24$, contrary to the hypothesis $q \ge 8$. It now follows from Theorem~\ref{thm:h} that $W_0(S) \ge 0$.
\end{proof}

Finally, the following result strongly restricts the values of $\a_1$ to consider.

\begin{lemma} Let $S$ be a numerical semigroup of depth $q \ge 4$ such that $|L| \le 12$. Then $\a_1 \le 2$.
\end{lemma}
\begin{proof} By Proposition~\ref{prop formulas}, we have
$|L| = q+(q-1)\a_1+(q-2)\a_2+\cdots+\a_{q-1}$.
Hence $|L| \ge q + (q-1)\a_1$. We have $q \ge 4$. If $\a_1 \ge 3$ then $|L| \ge 4+9=13$, contrary to the hypothesis on $|L|$.
\end{proof}

The cases $\a_1=2$, $1$ and $0$ will now be treated successively. We shall occasionally use the following notation.

\begin{notation} For all $i \ge 1$, we denote by $\ind_i=\ind_{A_i}$ the indicator function of $A_i$.
\end{notation}

\subsection{When $\a_1=2$}\label{a_1=2}

\begin{proposition} Let $S$ be a numerical semigroup of depth $q \ge 4$ such that $|L| \le 12$. If $\a_1=2$, then $q=4$ and $W_0(S) \ge 0$.
\end{proposition}
\begin{proof} Since $\a_1 = 2$, we have $12 \ge |L| \ge q + 2(q-1)=3q-2$. It follows that $q < 5$, whence $q=4$ since $q \ge 4$ by hypothesis. We have
\begin{equation*}
|P \cap L| \ge 3,  \ \ |L| = 10+ 2\a_2+ \a_3, \ \ |D_4| = 3+\a_2+\a_3+\a''_4.
\end{equation*}
Hence $2\a_2+ \a_3 \le 2$ and so $\a_2 \le 1$. 

\smallskip
\sbu If $\a_2=0$, then $|L|=10+\a_3$ and so $\a_3 \le 2$. Since $2A_1 \subset S_2 \cup S_3$, and since $A_2=\emptyset$ and $2S_3 \cap S_4 = \emptyset$, it follows that
$$
A_4 \cap D \subseteq (A_1+A_3).
$$
Hence $\a_4'' \le 2 \a_3$. Therefore
\begin{eqnarray*}
W_0(S) & \ge & 3(10+\a_3)-4(3+\a_3+2\a_3)+\rho \\
& = & 18-9\a_3+\rho \\
& \ge & \rho. 
\end{eqnarray*}

\smallskip
\sbu If $\a_2=1$, then $|L|=12+\a_3$ and so $|L|=12$ and $\a_3 =0$. Since $2A_1 \subset S_2 \cup S_3$ and $A_3=\emptyset$, it follows that
$$
A_4 \cap D \subseteq (A_1+A_2) \cup (2A_2).
$$
Hence $\a''_4 \le 2+1=3$. Therefore
\begin{eqnarray*}
W_0(S) & \ge & 3\cdot 12-4(3+1+3)+\rho \\
& = & 8 + \rho. \qedhere
\end{eqnarray*}
\end{proof}

\begin{remark} A better lower bound on $W_0(S)$ may be obtained by splitting $\a_i$ as $\a_i'+\a_i''$ for $i=2, 3$ in the above proof. For instance, we have only used $|P \cap L| \ge 3$. But we could have used $|P \cap L| \ge 4$ if either $\a'_2$ or $\a'_3$ were assumed positive, while if $\a'_2=\a'_3=0$, a sharper estimate on $\a''_4$ can been derived.
\end{remark}

\subsection{When $\a_1=1$}\label{case a1=1}
Since $12 \ge |L| \ge q +(q-1)=2q-1$, it follows that $q \le 6$. We shall successively treat the cases $q=4$, $5$ and $6$. Throughout Section~\ref{case a1=1}, we set
$$
A_1=\{x\}.
$$

\subsubsection{Case $q=4$}

Then $\a(S)=(1,\a_2,\a_3)$. We have $|L|=4+3+2\a_2+\a_3$, whence $2\a_2+\a_3 \le 5$, implying $\a_2 \le 2$. We successively examine the cases $\a_2=2,1,0$. To start with, we have
\begin{equation}
A_4 \cap D \subseteq (A_1+A_2) \cup (A_1+A_3) \cup (A_2+A_2) \cup (A_2+A_3).
\end{equation}

\noindent
\textbf{\underline{Subcase $\a_2=2.$}} Then $\a(S)=(1,2,\a_3)$. We have $|L|=11+\a_3$, whence $\a_3 \le 1$. Denote $$A_2=\{y_1, y_2\}.$$
Since $A_2 \cap D \subseteq 2A_1$, and since $|2A_1|=1$, we have $\a_2'' \le 1$ whence $\a_2' \in \{1,2\}$.

\smallskip
\sbu Assume first $\a_2'=1$. Say $y_1 \in P, y_2 \in D$. Then $y_2=2x$, and $|P \cap L|=3+\a_3'$.
We have $|D_4|=4+\a_3+\a_4''$, and
$$
(A_1+A_2) \cup (2A_2) = \{4x, 3x, 2x+y_1,x+y_1, 2y_1\}.
$$
Since $4x \notin S_3$ and since $\a_3 \le 1$, there are at most four possibilities for $A_3 \cap D$, listed below together with their consequences on $A_4 \cap D$. Note that Lemma~\ref{downset} plays a key role to deduce these consequences. For example, if $4x \in A_4$ or $2x+y_1 \in A_4$, then necessarily $3x \in A_3$ or $x+y_1 \in A_3$, respectively. Note also that the level function $\lambda=\lambda_S$ is nondecreasing. Consequently, in the last case $A_3=\{2y_1\}$ below, it follows that $x+y_1 \notin A_4$, for $\lambda(x+y_1) \le \lambda(2y_1)=3$ since $x < y_1$. Here then are the possibilities for $A_3 \cap D$:

\begin{enumerate}
\item If $A_3 \cap D = \emptyset$ then $A_4 \cap D \subseteq \{3x,x+y_1, 2y_1\}$. \vspace{-0.15cm}
\item If $A_3 \cap D = \{3x\}$ then $A_4 \cap D \subseteq \{4x,x+y_1, 2y_1\}$. \vspace{-0.15cm}
\item If $A_3 \cap D = \{x+y_1\}$ then $A_4 \cap D \subseteq \{3x,2x+y_1, 2y_1\}$. \vspace{-0.15cm}
\item If $A_3 \cap D = \{2y_1\}$ then $A_4 \cap D \subseteq \{3x, 2y_1\}$. 
\end{enumerate}
In either case, we have $\a_4'' \le 3$. Recall also that $\alpha_3 \le 1$ here. Hence
\begin{eqnarray*}
W_0(S) & = & (3+\alpha_3')(11+\alpha_3)-4(4+\a_3+\a_4'')+\rho\\
& \ge & (3+\alpha_3')(11+\alpha_3)-4(7+\a_3)+\rho\\
& \ge & 4 +11\alpha_3'+\alpha_3'\alpha_3+\rho.
\end{eqnarray*}

\sbu Assume now $\a_2'=2$, so that $y_1,y_2 \in P$. Then $|P \cap L|=4+\a_3'$. We have $|D_4|=4+\a_3+\a_4''$, and
$$
(2A_1) \cup (A_1+A_2) \cup (2A_2) = \{2x, x+y_1, x+y_2, 2y_1, y_1+y_2, 2y_2\}.
$$
Up to permutation of $y_1, y_2$, and using $\a_3 \le 1$ and $2x \notin S_4$, here are the possibilities for $A_3 \cap D$, together with their consequences for $A_4 \cap D$:
\begin{enumerate}
\item If $A_3 \cap D = \emptyset$ then $A_4 \cap D \subseteq \{x+y_1, x+y_2, 2y_1, y_1+y_2, 2y_2\}$. \vspace{-0.15cm}
\item If $A_3 \cap D = \{2x\}$ then $A_4 \cap D \subseteq \{3x, x+y_1, x+y_2, 2y_1, y_1+y_2, 2y_2\}$. \vspace{-0.15cm}
\item If $A_3 \cap D = \{x+y_1\}$ then $A_4 \cap D \subseteq \{x+y_2, 2y_1, y_1+y_2, 2y_2\}$. \vspace{-0.15cm}
\item If $A_3 \cap D = \{2y_1\}$ then $A_4 \cap D \subseteq \{x+y_1, x+y_2, y_1+y_2, 2y_2\}$. \vspace{-0.15cm}
\item If $A_3 \cap D = \{y_1+y_2\}$ then $A_4 \cap D \subseteq \{x+y_1, x+y_2, 2y_1, 2y_2\}$.
\end{enumerate}
In either case, we have $\a_4'' \le 6$. Hence
\begin{eqnarray*}
W_0(S) & = & (4+\alpha_3')(11+\alpha_3)-4(4+\alpha_3+\alpha_4'')+\rho \\
 & \ge &  (4+\alpha_3')(11+\alpha_3)-4(10+\alpha_3)+\rho \\
& = & 4+11\alpha_3'+\alpha_3'\alpha_3+\rho.
\end{eqnarray*}

\medskip
\noindent
\textbf{\underline{Subcase $\a_2=1.$}} Then $\a(S)=(1,1,\a_3)$. We have $|L|=9+\a_3$, whence $\a_3 \le 3$. We also have $|P \cap L| =2+\a_2'+\a_3'$ and $|D_4|=3+\alpha_3+\alpha_4''$. Denote $$A_2=\{y\}.$$

\sbu Assume first $y \in A_2 \cap D$. Then $y = 2x$, $|P \cap L| =2+\a_3'$ and $A_3 \cap D \subseteq \{3x\}$. Thus $\a_3'' \le 1$, and either 
$$A_4 \cap D \subseteq \{3x\} \cup (x+P_3) \,\,\, \textrm{ or }\,\,\, A_4 \cap D \subseteq \{4x\} \cup (2x+P_3) \cup (x+P_3).$$
We claim that $\a_4'' \le 1 + \a_3'$ in both cases. This is clear in the first one. In the second one, for all $z \in P_3$ we have
$$
|A_4 \cap \{2x+z, x+z\}| \le 1.
$$
Therefore $|A_4 \cap ((2x+P_3) \cup (x+P_3))| \le |P_3|,$ implying $\a_4'' \le 1 + \a_3'$ here as well. Using $\a_3 =\a_3'+\a_3'' \le \a_3'+1$, we have
\begin{eqnarray*}
W_0(S) & = & (2+\a_3')(9+\a_3)-4(3+\alpha_3+\alpha_4'')+\rho\\
& \ge & (2+\a_3')(9+\a_3)-4(4+\a_3'+\a_3)+\rho\\
& = & 2+5\a_3'-2\a_3+\a_3\a_3'+\rho\\
& \ge & 2+5\a_3'-2(\a_3'+1)+\a_3\a_3'+\rho\\
& \ge & 3\a_3'+\a_3\a_3'+\rho.
\end{eqnarray*}

\smallskip
\noindent
\sbu Assume now $y \in A_2 \cap P$. Hence $|P \cap L| =3+\a_3'$. Here we have
$$
A_3 \cap D \subseteq \{2x, x+y, 2y\}
$$
and so $\a_3'' \le 3$. Let us examine in turn the possibilities for $A_3 \cap D$ and their consequences for $A_4 \cap D$:
\begin{enumerate}
\item If $A_3 \cap D = \emptyset$ then $A_4 \cap D \subseteq \{x+y,2y\}\cup (\{x,y\}+P_3)$. \vspace{-0.15cm}
\item If $A_3 \cap D = \{2x\}$ then $A_4 \cap D \subseteq \{3x,x+y,2y\}\cup (\{x,y\}+P_3)$. \vspace{-0.15cm}
\item If $A_3 \cap D = \{x+y\}$ then $A_4 \cap D \subseteq \{2y\}\cup (\{x,y\}+P_3)$. \vspace{-0.15cm}
\item If $A_3 \cap D = \{2y\}$ then $A_4 \cap D \subseteq \{x+y,3y\}\cup (\{x,y\}+P_3)$. \vspace{-0.15cm}
\item If $A_3 \cap D = \{2x, x+y\}$ then $A_4 \cap D \subseteq \{3x,2x+y,2y\}\cup (\{x,y\}+P_3)$. \vspace{-0.15cm}
\item If $A_3 \cap D = \{2x, 2y\}$ then $A_4 \cap D \subseteq \{3x,x+y,2y\}\cup (\{x,y\}+P_3)$. \vspace{-0.15cm}
\item If $A_3 \cap D = \{x+y, 2y\}$ then $A_4 \cap D \subseteq \{x+2y,3y\}\cup (\{x,y\}+P_3)$. \vspace{-0.15cm}
\item If $A_3 \cap D = \{2x, x+y, 2y\}$ then $A_4 \cap D \subseteq \{3x,2x+y,x+2y,3y\}\cup (\{x,y\}+P_3)$. \vspace{-0.15cm}
\end{enumerate}
Note that $|\{x,y\}+P_3| \le 2|P_3|=2\a_3'$. Consequently, distinguishing between the first seven cases and the last one, we have
\begin{eqnarray*}
\a_3'' \le 2 & \Rightarrow & \a_4'' \le 3+2\a_3' \\
\a_3'' = 3 & \Rightarrow & \a_4'' \le 4+2\a_3'.
\end{eqnarray*}

\noindent
-- Assume first $\a_3'' \le 2$, so that $\a_3 \le \a_3'+2$ and $\a_4'' \le 3+2\a_3'$. We have
\begin{eqnarray*}
W_0(S) & = & (3+\a_3')(9+\a_3)-4(3+\a_3+\a_4'')+\rho\\
& \ge & (3+\a_3')(9+\a_3)-4(6+\a_3+2\a_3')+\rho\\
& = & 3+\a_3'-\a_3+\a_3'\a_3+\rho\\
& \ge & 1+\a_3'\a_3+\rho.
\end{eqnarray*}

\noindent
-- Assume now $\a_3'' = 3$, so that $\a_4'' \le 4+2\a_3'$. Since $\a_3 \le 3$, it follows that $\a_3=\a_3''=3$, whence $\a_3'=0$ and $\a_4'' \le 4$. We have
\begin{eqnarray*}
W_0(S) & = & (3+\a_3')(9+\a_3)-4(3+\a_3+\a_4'')+\rho\\
& \ge & 3\cdot 12-4\cdot 10+\rho\\
& = & -4+\rho.
\end{eqnarray*}
We now show that $\rho \ge 6$ here. Indeed, since $y \in A_2$ and $3y \in A_4$, we have
$$
2m-\rho \le y, \ \ 3y \le  5m-\rho-1.
$$
Therefore $3(2m-\rho) \le 3y \le 5m-\rho-1$, implying $m \le 2\rho-1$. Now, since $m = 1+\a_1+\a_2+\a_3+\a_4 \ge 10$ by~\eqref{eq sum a_i}, it follows that $\rho \ge 6$, whence $W_0(S) \ge 2$ and we are done.

\medskip
\noindent
\textbf{\underline{Subcase $\a_2=0.$}} Then $\a(S)=(1,0,\a_3)$ here, and incidentally $\a_3 \le 5$ since $|L| \le 12$. We have
\begin{equation*}
|P \cap L| = 2+\a_3',  \ \ |L| = 7+\a_3, \ \ |D_4| = 2+\a_3+\a_4''.
\end{equation*}

Recalling that $A_1=\{x\}$, we distinguish the cases where $2x \in A$ or not.

\smallskip
\noindent
\sbu Assume first $2x \in A$. Since $\lambda(2x) \in \{2,3\}$ and since $A_2=\emptyset$, it follows that $2x \in A_3$. Hence $\a_3''=1$ and so $\a_3=1+\a_3'$. Thus
\begin{equation*}
|L| = 8+\a_3', \ \ |D_4| = 3+\a_3'+\a_4''.
\end{equation*}
Since $A_4 \cap D \subseteq \{3x\} \cup P_3$, it follows that $\a_4'' \le 1+\a_3'$ and hence $|D_4| \le 4+2\a_3'$. A straightforward computation then yields
$$
W_0(S) \ge \rho.
$$

\smallskip
\noindent
\sbu Assume now $2x \notin A$. It follows that $A_3 \subseteq P$, i.e. $\a_3=\a_3'$, and $A_4 \cap D \subseteq x+A_3$. Therefore $\a_4'' \le \a_3$n and so $|D_4| = 2 + \a_3+\a_4'' \le 2+2\a_3$. This implies here
$$
W_0(S) \ge 6+\rho.
$$

This concludes the case $\a_1=1$ and $q=4$. 

\subsubsection{Case $q=5$}
We now tackle the case $\a_1=1$ and $q=5$, i.e. $\a(S)=(1,\a_2,\a_3,\a_4)$. We have $|L|=9+3\a_2+2\a_3+\a_4$, whence $3\a_2+2\a_3+\a_4 \le 3$, implying $\a_2 \le 1$. 

\medskip
\noindent
\textbf{\underline{Subcase $\a_2=1.$}} Then $\a(S)=(1,1,\a_3,\a_4)$. Then $|L|=12$ and $\a_3 = \a_4 =0$. Moreover, $|D_5|=\a_0+\a_1+\a_2+\a_5''=3+\a_5''$. Set 
$$A_2=\{y\}.$$ 

\sbu Assume first $y \in D$. Then $y=2x$. Since $A_3=A_4=\emptyset$ by hypothesis, it follows that $|P \cap L|=2$ and $A_5 \cap D \subseteq \{3x\}$. Therefore $\a_5'' \le 1$. We conclude that
\begin{eqnarray*}
W_0(S) & = & 2\cdot 12 - 5(3+\a_5'')+\rho \\
& \ge & 4+\rho.
\end{eqnarray*}

\sbu Assume now $y \in P$. Then $|P \cap L|=3$ and $A_5 \cap D \subseteq \{2y\}$ since $\lambda(2x) \le 3$ and $\lambda(x+y) \le 4$. Hence $\a_5'' \le 1$ and
\begin{eqnarray*}
W_0(S) & = & 3\cdot 12 - 5(3+\a_5'')+\rho \\
& \ge & 16+\rho.
\end{eqnarray*}

\medskip
\noindent
\textbf{\underline{Subcase $\a_2=0.$}} Then $\a(S)=(1,0,\a_3,\a_4)$ and $|L|=9+2\a_3+\a_4$. Hence $\a_3 \le 1$. Moreover, $|D_5|=\a_0+\a_1+\a_3+\a_4+\a_5''=2+\a_3+\a_4+\a_5''$. We have $A_2 = \emptyset$.

\medskip
\textbf{\underline{Subsubcase $\a_3=1.$}} Then $\a(S)=(1,0,1,\a_4)$. We have $|L|=11+\a_4$, hence $\a_4 \le 1$. Since $A_2 = \emptyset$, we have $2x \in A$ if and only if $2x \in A_3$. 

\sbu Assume first $2x \in A$. Then $A_3=\{2x\}$ since $A_2=\emptyset$ and $\a_3=1$.
Hence $A_4 \cap D \subseteq \{3x\}$ and $A_5 \cap D \subseteq \{3x,4x\} \cup (x+A_4\cap P)$. Therefore
\begin{equation}\label{a5''}
\a_4''=\ind_4(3x), \ \ \a_5'' \le \ind_5(3x)+\ind_5(4x)+\a_4'.
\end{equation}
Note also that if $4x \in A_5$, then $A_4=\{3x\}$ by Lemma~\ref{downset} and the bound $\a_4 \le 1$. We have
\begin{eqnarray*}
|P \cap L|=2+\a_4', \ \ |L|=11+\a_4, \ \ |D_5|=3+\a_4+\a_5''=3+\a_4'+\ind_4(3x)+\a_5''.
\end{eqnarray*}
A straightforward computation, using \eqref{a5''}, then yields
\begin{eqnarray*}
W_0(S) & = & (2+\a_4')(11+\a_4) - 5(3+\a_4'+\ind_4(3x)+\a_5'')+\rho \\
& \ge & (2+\a_4')(11+\a_4) - 5(3+2\a_4'+\ind_4(3x)+\ind_5(3x)+\ind_5(4x))+\rho \\
& \ge & 7-3\cdot \ind_4(3x)-5\cdot \ind_5(3x)-5\cdot \ind_5(4x)+\rho.
\end{eqnarray*}
If $4x \notin A_5$ then $W_0(S) \ge 2+\rho$ since $\ind_4(3x)+\ind_5(3x) \le 1$ and we are done. If $4x \in A_5$, then $3x \in A_4$ as noted above, whence $W_0(S) \ge -1+\rho$. But then $4x \in C$, since $\lambda(4x)=5$ whereas $\lambda(2x)=3$. Thus $\rho \ge |C| \ge 1$, implying $W_0(S) \ge 0$, as desired.

\sbu Assume now $2x \notin A$. Then $\a_3'=\a_3=1$ and $|P \cap L|=3+\a_4'$. We have $A_5 \cap D \subseteq (A_1+A_3) \cup (A_1+A_4)$, whence $\a_5'' \le 2$ since $\a_1=\a_3=1$ and $\a_4 \le 1$. Therefore $|D_5|=3+\a_4+\a_5'' \le 6$, so that
\begin{eqnarray*}
W_0(S) & = & (3+\a_4')(11+\a_4) - 5|D_5|+\rho \\
& \ge & 3\cdot 11 -30+\rho\\
& \ge & 3+\rho.
\end{eqnarray*}

\medskip

\textbf{\underline{Subsubcase $\a_3=0.$}} Then $\a(S)=(1,0,0,\a_4)$. We have $|L|=9+\a_4$, hence $\a_4 \le 3$ and $|D_5|=2+\a_4+\a_5''$. Since $A_2=A_3=\emptyset$, and since $2 \le \lambda(2x) \le 3$, it follows that $2x \notin A$ and $A_4=P_4$. Thus $A_5 \cap D \subseteq x+P_4$, so that $\a_5'' \le \a_4$ and $|P \cap L|=2+\a_4$. It follows that
\begin{eqnarray*}
W_0(S) & \ge & (2+\a_4')(9+\a_4) - 5(2+\a_4+\a_5'')+\rho \\
& \ge & (2+\a_4')(9+\a_4) - 5(2+2\a_4)+\rho \\
& \ge & 8+\a_4+\a_4^2+\rho.
\end{eqnarray*}

\smallskip
This concludes the case $\a_1=1$ and $q=5$.

\subsubsection{Case $q=6$}
Still for $\a_1=1$, we now tackle the last case $q=6$. Then $\a(S)=(1,\a_2,\a_3,\a_4, \a_5)$. We have $|L|=11+4\a_2+3\a_3+2\a_4+\a_5$, whence $\a_2=\a_3=\a_4=0$ and $\a_5 \le 1$. 

\medskip
\noindent
\textbf{\underline{Subcase $\a_5=1.$}} Then $\a(S)=(1,0,0,0,1)$. We have $|L|=12$ and $|D_6|= \a_0+\a_1+\a_5+\a_6''=3+\a_6''$. Recall that $A_1=\{x\}$. Since $\lambda(2x) \le 3$ and $A_2=A_3=A_4=\emptyset$, it follows that $2x \notin A$. Whence $A_5=P_5$, i.e. $\a_5'=\a_5=1$, and $A_6 \cap D \subseteq x+A_5$. Hence $lP \cap Ll=4$ and $\a_6'' \le 1$. Consequently $|D_6| \le 4$, and
\begin{eqnarray*}
W_0(S) & \ge & 3\cdot 12- 6 \cdot 4+\rho \\
& = & 12+\rho.
\end{eqnarray*}

\medskip
\noindent
\textbf{\underline{Subcase $\a_5=0.$}} Then $\a(S)=(1,0,0,0,0)$. We have $|L|=11$ and $|D_6|=2+\a_6''$. Since $A_1=\{x\}$ and $A_2,A_3,A_4,A_5$ are all empty, it follows that $|P \cap L|=2$ and $A_6 \cap D=\emptyset$, i.e. $\a_6''=0$. Thus
$$
W_0(S) = 2 \cdot 11 -6 \cdot 2+\rho = 10+\rho.
$$

\smallskip

Summarizing, we have shown that if $|L| \le 12$, $\a_1=1$ and $q=6$, then $W_0(S) \ge 10+\rho$. This concludes the case $|L| \le 12$ and $\a_1=1$.

\subsection{When $\a_1=0$}
As noted at the beginning of Section~\ref{main}, it suffices to consider the cases $4 \le q \le 7$.
\subsubsection{Case $q=4$} Then $\a(S)=(0,\a_2,\a_3)$. Since $\lceil q/2\rceil = 2$ here, Theorem~\ref{thm:h} yields $W_0(S) \ge 0$.

\subsubsection{Case $q=5$} Then $\a(S)=(0,\a_2,\a_3,\a_4)$. We have $12 \ge |L|=5+3\a_2+2\a_3+\a_4$. Hence $\a_2 \le 2$. We now examine successively the cases $\a_2=2,1,0$. Since $A_1=\emptyset$, it follows that $A_2 \subset P$.

\medskip
\noindent
\textbf{\underline{Subcase $\a_2=2.$}} Then $\a(S)=(0,2,\a_3,\a_4)$. We have $|L|=11+2\a_3+\a_4$. Hence $\a_3=0$ and $\a_4 \le 1$. Since $A_2 \subset P$, we have $\a_2'=\a_2=2$ here. Thus $|P \cap L| \ge 3$. Set
$$
A_2=\{x_1,x_2\}.
$$

\medskip
\underline{Assume first $\a_4=1$.} Then $\a(S)=(0,2,0,1)$, so that $|L|=12$ and $|D_5|=4+\a_5''$.
Denote $A_4=\{z\}$. 

\smallskip
\sbu If $z \in P$, then $|P \cap L|=4$ and $2A_2 \cap (A_3 \cup A_4)=\emptyset$.
Now
$$
A_5 \cap D \subseteq 2 A_2 \cup (A_2+A_4).
$$
Since $|2A_2| \le 3$ and $|A_2+A_4| \le |A_2|$, it follows that $\a_5'' \le 5$. Hence
\begin{eqnarray*}
W_0(S) & = & |P \cap L||L|-5|D_5|+\rho \\
& \ge & 4 \cdot 12 -5 \cdot 9+\rho \\
& = & 3+\rho
\end{eqnarray*}
and we are done if $z \in P$.

\smallskip
\sbu If $z \in D$, then $|P \cap L|=3$ and $z \in A_2$ since $A_3=\emptyset$. Hence, up to renumbering, either $z=2x_1$ or $z=x_1+x_2$. If $z=2x_1$ then
$A_5 \cap D \subseteq \{3x_1,x_1+x_2,2x_2\}$, whereas if $z=x_1+x_2$, then $A_5 \cap D \subseteq \{2x_1,2x_2\}$. In either case, we have $\a_5'' \le 3$. Therefore $|D_5|\le 7$. It follows that
\begin{eqnarray*}
W_0(S) & = & |P \cap L||L|-5|D_5|+\rho \\
& \ge & 3 \cdot 12 -5 \cdot 7+\rho \\
& = & 1+\rho
\end{eqnarray*}
and we are done as well if $z \in D$.

\medskip
\underline{Assume now $\a_4=0$.} Then $\a(S)=(0,2,0,0)$, so that $|L|=11$ and $|D_5|=3+\a_5''$. In that case, we have $A_5 \subseteq 2A_2$, whence $\a_5'' \le 3$ and $|D_5| \le 6$. It follows that
$$
W_0(S) \ge 3 \cdot 11 - 5 \cdot 6 +\rho \ge 3+\rho.
$$
That concludes the subcase $\a_2=2$, i.e. $\a(S)=(0,2,\a_3,\a_4)$ here.
 
\medskip
\noindent
\textbf{\underline{Subcase $\a_2=1.$}} Then $\a(S)=(0,1,\a_3,\a_4)$. We have $|L|=8+2\a_3+\a_4$. Hence $\a_3 \le 2$. Since $A_2 \subset P$, we have $\a_2'=\a_2=1$ here. Set
$$
A_2=\{x\}.
$$

\smallskip

\textbf{\underline{Subsubcase $\a_3=2.$}} Then $\a(S)=(0,1,2,\a_4)$. We have $|L|=12+\a_4$, whence $|L|=12$ and $\a_4=0$. Thus $\a(S)=(0,1,2,0)$. Hence $|D_5|=4+\a_5''$. Since $A_3 \cap D \subseteq 2A_2$, it follows that $\a_3'' \le 1$ and hence $\a_3' \in \{1,2\}$. Therefore $|P \cap L| \ge 3$. Set
$$
A_3=\{y_1,y_2\}.
$$

\sbu Assume first $\a_3'=1$. Hence $\a_3''=1$, and up to renumbering, we may assume $y_1=2x$, $y_2 \in P$. Hence $|P \cap L|=3$ and
$$
A_5 \cap D \subseteq \{3x, x+y_2, 2y_2\}.
$$
Therefore $\a_5'' \le 3$ and so $|D_5| \le 7$. We then have
\begin{eqnarray*}
W_0(S) & = & |P \cap L||L|-5|D_5|+\rho \\
& \ge & 3 \cdot 12 -5 \cdot 7+\rho \\
& = & 1+\rho
\end{eqnarray*}
and we are done here.

\smallskip
\sbu Assume now $\a_3'=\a_3=2$. Thus $y_1,y_2 \in P$ and so $|P \cap L|=4$. We have
$$
A_5 \cap D \subseteq \{2x\} \cup (x+A_3) \cup (2A_3 \cap A_5).
$$ 
Thus $\a_5'' \le 3 + |2A_3 \cap A_5|$, and of course $|2A_3 \cap A_5| \le |2A_3| \le 3$ since $|A_3|=2$. At this point we have
$$
|D_5| \le 7 + |2A_3 \cap A_5|.
$$
Therefore
\begin{eqnarray*}
W_0(S) & = & |P \cap L||L|-5|D_5|+\rho \\
& \ge & 4 \cdot 12 -5 \cdot (7+|2A_3 \cap A_5|)+\rho \\
& = & 13-5|2A_3 \cap A_5|+\rho.
\end{eqnarray*}
Now $\rho \ge |2A_3 \cap A_5|$, since $2A_3 \cap A_5 \subseteq C$ and $|C| \le \rho$. Hence
$$
W_0(S) \ge 13-4|2A_3 \cap A_5|.
$$
But $|2A_3 \cap A_5| \le 3$ as seen above. It follows that $W_0(S) \ge 1$ and we are done here.

\medskip

\textbf{\underline{Subsubcase $\a_3=1.$}} Then $\a(S)=(0,1,1,\a_4)$. We have
\begin{equation}\label{general}
|P \cap L| = 2+\a_3'+\a_4',  \ \ |L| = 10+\a_4, \ \ |D_5| = 3+\a_4+\a_5''.
\end{equation}
We examine four subcases, depending on which multiples of $x$ belong to $A$.

\smallskip
\textbf{Case 1. $2x \in A_3, 3x \in A_4$.} Then $A_3=\{2x\}$ since $\a_3=1$, and $A_4 \cap D = \{3x\}$. Hence $\a_4=\a_4'+1$. Moreover, 
$$A_5 \cap D \subseteq \{4x\} \cup (x+A_4 \cap P).$$
Whether $4x$ belongs to $A_5$ or not will be measured by $\ind_5(4x) \in \{0,1\}$, where $\ind_i=\ind_{A_i}$ denotes the indicator function of $A_i$ for all $i$. Thus, 
$$
\a_5'' \le \ind_5(4x)+\a_4',
$$
and so
\begin{equation}\label{data}
|P \cap L| = 2+\a_4',  \ \ |L| = 11+\a_4', \ \ |D_5| = 4+\a_4' +\a_5'' \le 4+2\a_4'+\ind_5(4x).
\end{equation}
Plugging this data into \eqref{W0}, we get
$$
W_0(S) \ge 2+3\a_4'+\a_4'^2-5\cdot \ind_5(4x)+\rho.
$$
If either $\a_4' \not=0$ or $\ind_5(4x)=0$, then $W_0(S) \ge \rho$ and we are done. If $\a_4'=0$ and $\ind_5(4x)=1$, then $W_0(S) \ge -3+\rho$. But in this case, we claim that $\rho \ge 3$. Indeed, both $2x,3x$ are \emph{compressed} Apéry elements, since $\lambda(x) = 2, \lambda(2x) = 3, \lambda(3x) =4$ and $2x=x+x, 3x=x+2x$. Moreover, since $\ind_5(4x)=1$ here, then $4x$ is also a compressed Apéry element, using $\lambda(4x) =5$ and $4x=2x+2x$. Hence $\{2x,3x,4x\} \subseteq C$. Since $\rho \ge |C|$, this proves $\rho \ge 3$ as claimed and yields $W_0(S) \ge 0$, as desired.

\smallskip
\textbf{Case 2. $2x \in A_3, 3x \notin A_4$.} Then $A_3=\{2x\}$ and $A_4 \subset P$, so that $\a_4=\a_4'$. We have
$$
A_5 \cap D \subseteq \{3x\} \cup (x+A_4).
$$
Hence $\a_5'' \le 1+\a_4$, and so 
\begin{equation*}
|P \cap L| = 2+\a_4,  \ \ |L| = 10+\a_4, \ \ |D_5| = 3+\a_4 +\a_5'' \le 4+2\a_4.
\end{equation*}
Hence $W_0(S) \ge 2\a_4+\a_4^2+\rho$ and we are done in the present case.

\smallskip
\textbf{Case 3. $2x \in A_4$.} Then $A_3=A_3 \cap P$, so that $\a_3'=\a_3=1$. We set $A_3=\{y\}$. We have $\{2x\} \subseteq A_4 \cap D \subseteq \{2x,x+y\}$ and $A_5 \cap D \subseteq \{3x,2y,x+y\}$. Since $x+y$ cannot belong to both $A_4, A_5$, it follows that $\a_4''+\a_5'' \le 4$.

Assume first $\a_4=2$. We then have
\begin{equation}
|P \cap L| = 3+\a_4',  \ \ |L| = 12, \ \ |D_5| = 3+\a_4'+\a_4'' +\a_5'' \le 7+\a_4'.
\end{equation}
It follows that $W_0(S) \ge 12(3+\a_4') - 5(7+\a_4')+\rho \ge 1+\rho$ and we are done.

Assume now $\a_4=1$. Then $A_4=\{2x\}$ and $|L|=11$. If $\a_5'' \le 2$, then
\begin{equation*}
|P \cap L| = 3,  \ \ |L| = 11, \ \ |D_5| = 4 +\a_5'' \le 6.
\end{equation*}
It follows that $W_0(S) \ge 3 \cdot 11 - 5 \cdot 6 +\rho \ge 3+\rho$ and we are done. But if $\a_5''=3$, i.e. if $A_5 \cap D = \{3x,2y,x+y\}$, then we only get $W_0(S) \ge -2+\rho$. But in this case, since $3x,2y \in C$, it follows that $\rho \ge 2$ whence $W_0(S) \ge 0$ as desired.

\smallskip
\textbf{Case 4. $2x \notin A$.} Then as above, $A_3=A_3 \cap P$, so that $\a_3'=\a_3=1$, and we set $A_3=\{y\}$. If $x+y \in A$, then it belongs to either $A_4$ or $A_5$. But in any case, $x+2y \notin A$ since $\lambda(x+2y) \ge 6$, as easily seen. It follows that $A_4 \cap D \subseteq \{x+y\}$ and $A_5 \cap D \subseteq \{x+y,2y\}$. Hence
$$
\a_4=\a_4'+\ind_4(x+y), \ \ \a_5'' \le 1+\ind_5(x+y).
$$
It follows that
\begin{equation*}
|P \cap L| = 3+\a_4',  \ \ |L| = 10+\a_4'+\ind_4(x+y), \ \ |D_5| \le 4+\a_4'+\ind_4(x+y)+\ind_5(x+y).
\end{equation*}
A straightforward computation, using $\ind_4(x+y)+\ind_5(x+y) \le 1$, then yields
\begin{eqnarray*}
W_0(S) & \ge & 10 +8\a_4'+\a_4'^2+(\a_4'-2)\cdot \ind_4(x+y)-5\cdot \ind_5(x+y)+\rho \\
& \ge & 5+\rho
\end{eqnarray*}
and we are done.

\medskip

\textbf{\underline{Subsubcase $\a_3=0.$}} Then $\a(S)=(0,1,0,\a_4)$. We have
\begin{equation*}
|P \cap L| = 2+\a_4',  \ \ |L| = 8+\a_4, \ \ |D_5| = 2+\a_4+\a_5''.
\end{equation*}
Here $A_3=\emptyset$, and $\a_4 \le 4$ since $|L| \le 12$ by assumption.

\smallskip
\textbf{Case 1. $2x \in A$.} Then $2x \in A_4 \sqcup A_5$. Thus 
$$A_4\cap D \subseteq \{2x\}, \ \ A_5\cap D \subseteq \{2x,3x\} \cup (x+A_4\cap P).$$
Hence $\a_4=\a_4'+\a_4''=\a_4'+\ind_4(2x)$ and $\a_5'' \le \ind_5(2x)+\ind_5(3x)+\a_4'$. We have
\begin{equation*}
|P \cap L| = 2+\a_4',  \ \ |L| = 8+\a_4'+\ind_4(2x), \ \ |D_5| \le 2+2\a_4'+\ind_4(2x)+\ind_5(2x)+\ind_5(3x).
\end{equation*}
A straightforward computation then yields
$$
W_0(S) \ge 6-3\cdot \ind_4(2x)-5\cdot \ind_5(2x)-5\cdot \ind_5(3x)+\rho.
$$
If $3x \notin A$, then $\ind_5(3x)=0$ and $W_0(S) \ge 1+\rho$. The case $3x \in A$ is more delicate. It implies $3x \in A_5$, i.e. $\ind_5(3x)=1$, and then of course $2x \in A_4$, i.e. $\ind_4(2x)=1$ and $\ind_5(2x)=0$. Thus $W_0(S) \ge -2+\rho$. It remains to show $\rho \ge 2$. Since $(\lambda(x), \lambda(2x), \lambda(3x)) = (2,4,5)$ here, it follows that $3x \in C$, whence $\rho \ge 1$. This is not strong enough yet. However, since $(\lambda(2x), \lambda(3x)) = (4,5)$, we have
$$
4m-\rho \le 2x, \ \ 3x \le  6m-\rho-1.
$$
Therefore $3(4m-\rho)/2 \le 3x \le 6m-\rho-1$, implying $\rho \ge 2$ as desired, and hence $W_0(S) \ge 0$.

\smallskip
\textbf{Case 2. $2x \notin A$.} Then $A_3=A_4 \cap D=\emptyset$. Hence $A_4=A_4 \cap P$ and $A_5 \cap D \subseteq x+A_4$. That is, we have $\a_4=\a_4'$ and $\a_5'' \le \a_4$. Thus $|D_5| \le 2+2\a_4$, and a straightforward computation yields
$$
W_0(S) \ge 6 + \rho.
$$

\subsubsection{Case $q=6$} Then $\a(S)=(0,\a_2,\a_3,\a_4,\a_5)$. We have $
12 \ge |L| = 6+4\a_2+3\a_3+2\a_4+\a_5$, whence $\a_2 \le 1$.

\medskip
\noindent
\textbf{\underline{Subcase $\a_2=1.$}} Then $|L|=10+3\a_3+2\a_4+\a_5$, implying $\a_3=0$ and $\a_4 \le 1$. Thus $\a(S)=(0,1,0,\a_4,\a_5)$. We set
$$
A_2 = \{x\},
$$
with $x \in P$ since $A_2 \cap D = \emptyset$. We have $A_3=\emptyset$.

\medskip

\textbf{\underline{Subsubcase $\a_4=1.$}} Then $|L|=12$, whence $\a_5=0$ and so $\a(S)=(0,1,0,1,0)$. Set $A_4=\{z\}$. Since $A_3=A_5=\emptyset$, it follows that $2x \in A$ if and only if $2x=z$. That is, either $z \in A\cap P$ and $2x \notin A$, or else $z=2x$. In either case, we have
$$
A_6 \cap D \subseteq x+A_4,
$$
whence $\a_6'' \le 1$. Summarizing, we have
\begin{equation*}
|P \cap L| = 2+\a_4',  \ \ |L| = 12, \ \ |D_6| = 3+\a_6'' \le 4.
\end{equation*}
It follows that
$$
W_0(S) \ge (2+\a_4')\cdot 12 - 6\cdot 4 + \rho \ge \rho 
$$
and we are done here.

\medskip

\textbf{\underline{Subsubcase $\a_4=0.$}} Then $\a(S)=(0,1,0,0,\a_5)$. Since $\lambda(2x) \in [3,5]$ and $A_3=A_4=\emptyset$, we have $2x \in A$ if and only if $2x \in A_5$. Therefore $A_5 \cap D \subseteq \{2x\}$, i.e. $\a_5'' \le 1$. However $3x \notin A_6$. For if $3x \in A_6$, then $2x \in A_5$ by Lemma~\ref{downset}, whence
$$
5m-\rho \le 2x, \ \ 3x \le 7m-\rho-1.
$$
Thus $3(5m-\rho)/2 \le 3x \le 7m-\rho-1$, whence $\rho \ge m+2$, in contradiction with $\rho \in [0,m[$. Therefore $A_6 \cap D \subseteq x+(A_5 \cap P)$, i.e. $\a_6'' \le \a_5'$. Summarizing, we have
\begin{equation*}
|P \cap L| = 2+\a_5',  \ \ |L| = 10+\a_5, \ \ |D_6| = 2+\a_5+\a_6'' \le 2+\a_5+\a_5'.
\end{equation*}
A straightforward computation then yields
$$
W_0(S) \ge 8 - 4\a_5''+\rho,
$$
whence $W_0(S) \ge 4+\rho$ since $\a_5'' \le 1$.

\medskip
\noindent
\textbf{\underline{Subcase $\a_2=0.$}} Then $\a(S)=(0,0,\a_3, \a_4, \a_5)$ and Theorem~\ref{thm:h} yields $W_0(S) \ge 0$.

\medskip

\subsubsection{Case $q=7$} Here $\a(S)=(0,\a_2,\a_3,\a_4,\a_5,\a_6)$. Then $12 \ge |L| \ge 7+5\a_2$, whence $\a_2 \le 1$. 

\smallskip
\sbu If $\a_2=1$, then $|L|=12$ and $\a_3=\dots=\a_6=0$. It follows that $A \cap D=\emptyset$, whence $\a_7''=0$. Thus $|P \cap L|=2$ and $|D_7| = 2$. Therefore
$$
W_0(S) = 2 \cdot 12 - 7\cdot 2+\rho=10+\rho
$$
and we are done in this case.

\smallskip
\sbu If $\a_2=0$, then $|L|=7+4\a_3+3\a_4+2\a_5+\a_6$, whence $\a_3 \le 1$. 

\smallskip
\sbbu Assume $\a_3=1$. Then $\a(S)=(0,0,1,\a_4,\a_5,\a_6)$ and $|L|=11+2\a_5+\a_6$, whence $\a_4=\a_5=0$ and $\a_6 \le 1$. Thus in fact, $\a(S)=(0,0,1,0,0,\a_6)$ and $|L|=11+\a_6$. We set
$$
A_3=\{x\}.
$$
Of course $x \in P$ since $A_1=A_2=\emptyset$. Since $A_4=A_5=\emptyset$, it follows that $A_6 \cap D \subseteq \{2x\}$. Moreover, since $(S_3+S_6)\cap S_7=\emptyset$, we also have $A_7 \cap D \subseteq \{2x\}$. Thus
$$
\a_6''=\ind_6(2x), \ \ \a_7''=\ind_7(2x).
$$
Summarizing, we have
\begin{equation*}
|P \cap L| = 2+\a_6',  \ \ |L| = 11+\a_6, \ \ |D_7| = 2+\a_6+\a_7'' = 2+\a_6'+\ind_6(2x)+\ind_7(2x).
\end{equation*}
A straightforward computation, using $\ind_6(2x)+\ind_7(2x) \le 1$, then yields
$$
W_0(S) \ge 1+\rho
$$
and we are done.

\smallskip
\sbbu Finally, assume $\a_3=0$. Then $\a(S)=(0,0,0,\a_4,\a_5,\a_6)$. In that case, since $\a_1=\dots=\a_{h-1}=0$ where $h=\lceil q/2 \rceil = 4$, Theorem~\ref{thm:h} implies $W_0(S) \ge 0$.

\medskip

\subsubsection{Case $q \ge 8$}

For depth $q \ge 8$, Proposition~\ref{prop q ge 8} yields $W_0(S) \ge 0$. 

\subsection{Concluding remarks} Having examined above all possible Apéry profiles $\a(S)$ compatible with $|L| \le 12$, the proof of Theorem~\ref{main} is now complete. 

\smallskip
As mentioned earlier, among the more than $10^{13}$ numerical semigroups of genus $g \le 60$, exactly five of them satisfy $W_0(S) < 0$. These five exceptions have depth $q=4$ and Apéry profile $\a(S)=(2,0,3)$, yielding $|L|=4+2\cdot 3+3=13$. Moreover, they satisfy $c=4m$, $|P \cap L|=|P_1|=3$ and $\a_4''=4$, whence $|D_4|=10$ and $W_0(S) = 3\cdot 13 - 4\cdot 10=-1$.

Infinite families of numerical semigroups satisfying $W_0(S)<0$ have been constructed in \cite{D,Jean}. However, they all turn out to satisfy $W(S) \ge 0$. It would be very interesting to characterize those $S$ satisfying $W_0(S)<0$, but that will likely be hard to achieve. Less ambitiously, can one determine how many such cases occur in, say, genus $g \le 100$?

\medskip
Authors' addresses:
\begin{itemize}
\item Shalom Eliahou, Univ. Littoral C\^ote d'Opale, UR 2597 - LMPA - Laboratoire de Math\'ematiques Pures et Appliqu\'ees Joseph Liouville, F-62228 Calais, France and CNRS, FR2037, France. \\
\texttt{eliahou@univ-littoral.fr}
\item D. Mar\'in-Arag\'on, Univ. Cádiz, Facultad de Ciencias Campus Universitario Río San Pedro s/n. 11510 Puerto Real, Cádiz, España \\
\texttt{daniel.marin@uca.es}
\end{itemize}

\end{document}